\documentclass[a4paper,reqno,11pt]{amsart} 
	\usepackage[latin1]{inputenc}
    \usepackage[T1]{fontenc}
    \usepackage[english]{babel}
    \usepackage{appendix}
    \usepackage[english]{minitoc}
    \usepackage[nice]{nicefrac}

\usepackage[top=2cm, bottom=2cm, left=2cm, right=2cm]{geometry}
\numberwithin{equation}{section}  \makeatletter\@addtoreset{equation}{section}
	\usepackage{amsmath}
    \usepackage{amssymb}
    \usepackage{amsbsy} 
    \usepackage{latexsym, amsfonts}
    \usepackage{graphics}
    \usepackage{ulem}
    \usepackage{hhline}
    \usepackage{dsfont}
    \usepackage{mathrsfs}
    \usepackage{color}
    \usepackage{fancyhdr}
    \usepackage{rotating}
    \usepackage{fancybox}
    \usepackage{colortbl}
    \usepackage{pifont}
    \usepackage{setspace}
    \usepackage{enumerate}
    \usepackage{multicol}
    \usepackage{varioref}
    \usepackage{lmodern}
    \usepackage{textcomp}
    \usepackage{euscript}
    \usepackage[pdftex]{hyperref}
			\newtheorem{theorem}{Theorem}[section]
			\newtheorem{definition}[theorem]{Definition}
			
			\newtheorem{lemma}[theorem]{Lemma}

			\newtheorem{corollary}[theorem]{Corollary}
			\newtheorem{remark}[theorem]{Remark}

\usepackage{xcolor}

\newcommand{\C}{\mathbb C}    \newcommand{\R}{\mathbb R}   \newcommand{\Z}{\mathbb Z} 	 
\newcommand{\Hq}{\mathbb H} \newcommand{\Hqr}{\widetilde{\mathbb{H}}} \newcommand{\Sq}{\mathbb S}
\newcommand{\bq}{\overline{q} } \newcommand{\Cn}{C_n(I)}
\newcommand{\norm}[1]{{\left\|{#1}\right\|}}    
   \newcommand{\scal}[1]{{\left\langle{#1}\right\rangle}}
    \newcommand{\bz}{\overline{z}}  \newcommand{\bw}{\overline{w}}

\begin{document}
\title[Generalized quaternionc Bargmann-Fock spaces and Segal-Bargmann transforms]{Generalized quaternionic Bargmann-Fock spaces and associated Segal-Bargmann transforms}

\thanks{The authors are partially supported by the Hassan II Academy of Sciences and Technology.
The research work of A.G. was partially supported by a grant from the Simons Foundation.}

\author{Amal El Hamyani}  
\author{Allal Ghanmi}    
\dedicatory{Dedicated to Professor Ahmed Intissar on the occasion of his 65th birthday }
\address{ Analysis and Spectral Geometry (A.G.S.),\newline
          Laboratory of Mathematical Analysis and Applications (L.A.M.A.),\newline
          Department of Mathematics, P.O. Box 1014,  Faculty of Sciences,\newline
          Mohammed V University in Rabat, Morocco}
\email{amalelhamyani@gmail.com}
 \email{ag@fsr.ac.ma}

\begin{abstract}
We introduce new classes of right quaternionic Hilbert spaces of Bargmann-Fock type $\mathcal{GB}_{m}^{2}(\Hq)$, labeled by nonnegative integer $m$,
generalizing the so-called slice hyperholomorphic Bargmann-Fock space introduced by Alpay, Colombo, Sabadini and Salomon
in \cite{AlpayColomboSabadiniSalomon2014}. They are realized as $L^2$-eigenspaces of a sliced second order differential operator.
The concrete description of these spaces is investigated and involves the so-called quaternionic Hermite polynomials. Their basic properties are discussed and the explicit formulae of their reproducing kernels are given.
Associated Segal-Bargmann transforms, generalizing the one considered by Diki and Ghanmi in \cite{DG2017}, are also introduced and studied. Connection to the quaternionic Fourier-Wigner transform is established.
\end{abstract}

\maketitle
\section{Introduction}

The classical Bargmann-Fock space is defined as the space of all $e^{-|z|^2} dxdy$-square integrable holomorphic functions on the complex plane.
This is a phase space which is known to be unitary isomorphic to the quantum mechanical configuration space $L^2(\R;dx)$ by means of the classical Segal-Bargmann transform (see for examples \cite{Bargmann1961,Folland1989,Zhu2012}).
A quaternionic analogue of this space in the context of slice hyperholomorphic functions of one quaternionic variable is introduced in \cite{AlpayColomboSabadiniSalomon2014} as
  \begin{align}\label{SHBFspace}
   \mathcal{F}^{2}_{slice}(\Hq) = \mathcal{SR}(\Hq) \cap L^2(\C_I;e^{-|q|^2}d\lambda_I),
   \end{align}
  where $ \mathcal{SR}(\Hq)$ denotes the space of all slice (left) regular $\Hq$-valued functions on $\Hq$ and $\C_I$; $I\in \mathbb{S}=\lbrace{q\in{\Hq};q^2=-1}\rbrace$, is a slice in $\Hq$.
It is shown there that $\mathcal{F}^{2}_{slice}(\Hq)$ is independent of $I$ and is a reproducing quaternionic Hilbert space.
A quaternionic analogue of the classical Segal-Bargmann is recently introduced and connect the slice hyperholomorphic Bargmann-Fock space $\mathcal{F}^{2}_{slice}(\Hq)$ to the classical $L^2$-Hilbert space of quaternionic-valued functions on the real line (see \cite{DG2017} for details).

In the present paper, we extend the notion of slice Bargmann-Fock space to the context of non hyperholomorphic functions.
The new space $\mathcal{GB}_{m}^{2}(\Hq)$ labeled by nonnegative integer $m=0,1,2, \cdots$, is connected to the specific $L^2$-eigenspaces
\begin{equation}\label{37}
\mathcal{F}_{m}^{2}(\Hq)= \left\{ f \in L^{2}(\Hq;e^{-\mid q\mid^{2}}d\lambda); \, \Delta_{q}f=m f\right\},
\end{equation}
of the  second-order differential operator
\begin{equation}\label{Laplaceian}
\Delta_{q}=\dfrac{-\partial^{2}}{\partial q\partial \overline{q}}+\overline{q}\dfrac{\partial}{\partial \overline{q}},
\end{equation}
The spaces $\mathcal{GB}_{m}^{2}(\Hq)$, defined on the whole $\Hq$, are specific subspaces of $\mathcal{F}_{m}^{2}(\Hqr)$.
It will be shown that the particular case of $\mathcal{F}_{0}^{2}(\Hq)$ and $\mathcal{GB}_{0}^{2}(\Hq)$, corresponding to $m=0$, gives rise to the full hyperholomorphic Bargmann-Fock space
  \begin{align}\label{fullHBFspace}
   \mathcal{F}^{2}_{full}(\Hq) = \mathcal{SR}(\Hq) \cap L^2(\Hq;e^{-\nu|q|^2}d\lambda)
   \end{align}
and the slice hyperholomorphic Bargmann-Fock space $\mathcal{F}^{2}_{slice}(\Hq)$ given through \eqref{SHBFspace}, respectively.
The $\mathcal{GB}_{m}^{2}(\Hq)$ are then called generalized quaternionic Bargmann-Fock spaces.
Our main  purpose is to give a concrete description of the spaces $\mathcal{F}_{m}^{2}(\Hqr)$ and $\mathcal{GB}_{m}^{2}(\Hq)$. We show that the $\mathcal{GB}_{m}^{2}(\Hq)$ is a reproducing kernel quaternionic Hilbert space. The expression of the corresponding reproducing kernel is given explicitly in Theorem \ref{thm:RepKernel}.
This was possible by solving the partial differential equation of hypergeometric type arising from the (right) eigenvalue problem
$\Delta_{q}f= f\mu $ on $\Hqr:=\mathbb{H}\setminus{\R}$, and manipulating the asymptotic behaviour of the involved confluent hypergeometric function.
We show in particular that the spectrum of $\Delta_{q}$ is purely discrete and constituted of the eigenvalues $\mu=\mu_m= m$ (Landau levels) occurring with infinite degeneracy (see Theorem \ref{thm:description}).
A concrete description of the elements of $\mathcal{F}_{m}^{2}(\Hqr)$ in $\mathcal{C}^\infty$ and $L^2$ pictures is given respectively by Theorems \ref{thm:CInftydescription} and \ref{thm:description}.
In such description, the so-called quaternionic Hermite polynomials play a crucial role. Such polynomials are introduced in \cite{Thiru2,Thiru1} and
an accurate systematic study of them can be found in \cite{El Hamyani}.

Associated Segal-Bargmann transform is then introduced and studied in some details. It generalizes the one considered in \cite{DG2017}.
The corresponding kernel function involves the real Hermite polynomial $H_{m}$ (see Theorem \ref{thm:KernelFctSBTm}).
As basic result, we show that this transform maps isometrically the $L^2$-Hilbert space of left-sided quaternionic-valued functions on the real line onto
$\mathcal{GB}_{m}^{2}(\Hq)$ (see Theorem \ref{thm:isometrySBTm}). The connection to a
quaternionic Fourier-Wigner transform is also established (see Theorem \ref{thm:FourierWigner-SBT}).

The rest of the paper is structured as follows
\begin{itemize}
\item Preliminaries.
\item Discussion of the problematic.
\item $\mathcal{C}^\infty$ and $L^{2}$-concrete spectral analysis of the operator $\Box_{q}$ on $\Hqr$. 
\item Generalized quaternionic Bargmann spaces and their reproducing kernels.
\item Generalized quaternionic Bargmann transforms $\mathcal{B}_m$.
\item Left-sided quaternionic Fourier-Wigner transform.
\end{itemize}

\section{Preliminaries} 

 We denote by $\Hq$ the divisor algebra of real quaternions. The standard basis $\lbrace{1,i,j,k}\rbrace$ satisfies the Hamiltonian multiplication $\mathbf{i}^2=\mathbf{j}^2=\mathbf{k}^2=\mathbf{i}\mathbf{j}\mathbf{k}=-1$, $\mathbf{i}\mathbf{j}=-\mathbf{j}\mathbf{i}=\mathbf{k}$, $\mathbf{j}\mathbf{k}=-\mathbf{k}\mathbf{j}=\mathbf{i}$ and $\mathbf{k}\mathbf{i}=-\mathbf{i}\mathbf{k}=\mathbf{j}.$
 The algebraic representation of a quaternion $q$ is $q=x_0+x_1\mathbf{i}+x_2\mathbf{j}+x_3\mathbf{k} \in{\Hq}$ with $x_0,x_1,x_2,x_3\in{\mathbb{R}}$.
  Accordingly, the quaternionic conjugate is defined to be $\overline{q}=x_0-x_1i-x_2j-x_3k=Re(q)-Im(q)$, so that
 $\overline{ pq }= \overline{q}\, \overline{p}$ for $p,q\in \Hq$. The modulus of $q$ is defined to be
 $$|q|=\sqrt{q\overline{q}}=\sqrt{x_0^2+x_1^2+x_2^2+x_3^2}.$$ 
While, with respect to the spherical coordinates
$$\left\{
  \begin{array}{ll}
    x_{0}=r\cos\theta, \\
   x_{1}=r\sin\theta\sin\phi\cos\psi,\\
    x_{2}=r\sin\theta\sin\phi\sin\psi,  \\
    x_{3}=r\sin\theta\cos\phi.
\end{array}
\right.
$$
with $r=|q|\in [0,+\infty[$, $\psi ,\phi \in [0,\pi]$ and $\theta \in [0,2\pi]$,
 the polar representation is given by
\begin{align} \label{polarRep}
q=re^{I\theta},
\end{align}
where $I$ is given by
 $
I= \sin\phi\cos\psi \mathbf{i} +\sin\phi\sin\psi \mathbf{j} + \cos\phi \mathbf{k}
$ 
and belongs to the unit sphere $\mathbb{S}=\lbrace{q\in{Im\Hq}; \vert{Im(q)}\vert=1}\rbrace$ in $Im\Hq$.
Notice for instance that $\mathbb{S}$ can be identified with the set of imaginary units $\mathbb{S}=\lbrace{q\in{\Hq};q^2=-1}\rbrace$.

Another interesting representation of $q\in \Hq$ is given by $q=x+I y$ for some real numbers $x$ and $y$ and imaginary unit $I\in \mathbb{S}$.
 Such decomposition is unique for any $q\in \Hq\setminus \R$ with $y>0$.
For every fixed $I\in{\mathbb{S}}$, the slice $L_I=\C_I := \mathbb{R}+\mathbb{R}I$ is isomorphic to the complex plane $\C$
so that it can be considered as a complex plane in $\Hq$ passing through $0$, $1$ and $I$.
Thus, $\Hq$ can be seen as the infinite union of complex planes, the slices. Their intersection is the real line $\mathbb{R}$.

The last representation is the basic idea in developing the theory of quaternionic slice hyperholomorphic functions that has been introduced
by Gentilli and Struppa in the seminal work \cite{GentiliStruppa07}. 
Since then, they  have been object of intensive research and the corresponding hyper-complex analysis theory have been developed. It has found many interesting applications in operator theory,
quantum physics, Schur analysis \cite{ColomboSabadiniStruppa2011,AlpayColomboSabadini2012,AlpayColomboSabadini2013,
GentiliStoppatoStruppa2013,AlpayBolotnikovColomboSabadini2016}.
The interesting readers can refer to  \cite{GentiliStruppa07,ColomboSabadiniStruppa2011,GentiliStoppatoStruppa2013,
CP,ColomboSabadiniStruppa2016} for more details.
According to \cite{GentiliStruppa07}, the left slice derivative $\partial_{s}f$ of a given real differential
quaternionic-valued function $f$ on a given open domain $\Omega\subset \Hq$ is defined by
\begin{equation}\label{12}
\partial_{s}f(q)=\left\{
                   \begin{array}{ll}
                     \dfrac{1}{2}\left(\dfrac{\partial f_{I_{q}}}{\partial x}-I_{q}\dfrac{\partial f_{I_{q}}}{\partial y}\right)(q), & \hbox{if $\quad q=x_{q}+y_{q}I_{q} \in \Omega \setminus\R$;} \\
                     \dfrac{df}{dx}(x_{q}), & \hbox{if $\quad q=x_{q}\in \Omega\cap \R$},
                   \end{array}
                 \right.
\end{equation}
 where $f_I(q=x+Iy)$, denotes the restriction of $f$ to $\Omega_{I}:= \Omega \cap L_{I}$.
 Thus, $f$ is said to be slice (left) regular, if it is a real differentiable on $\Omega$ and its restriction $f_I$ is holomorphic on $\Omega_I$ for every $I\in \mathbb{S}$. That is it has continuous partial derivatives with respect to $x$ and $y$ and the function $\overline{\partial_I} f : \Omega_I \longrightarrow \Hq$ defined by
$$
\overline{\partial_I} f(x+Iy) := 
\dfrac{1}{2}\left(\frac{\partial }{\partial x}+I\frac{\partial }{\partial y}\right)f_I(x+yI)
$$
 vanishes identically on $\Omega_I$.
The corresponding space, denoted $\mathcal{SR}(\Omega)$, is endowed with the natural uniform convergence on compact sets. It turns out that $\mathcal{S}\mathcal{R}(\Omega)$ is a right vector space over the noncommutative field $\Hq$ containing the power series $\sum_n q^na_n$; $a_n\in\Hq$, defined in its domain of convergence, which is proved to be an open ball $B(0,R):= \{q\in \Hq; \, |q| < R\}$.
Conversely, every given $\Hq$-valued slice regular function $f$ on $B(0,R)\subset \Hq$ has the following series expansion (\cite{GentiliStruppa07})
\begin{align}\label{expansion}
f(q)=\sum_{n=0}^{+\infty} q^{n} a_n; \qquad a_n = \frac{1}{n!}\frac{\partial^{n}f}{\partial x^{n}}(0).
\end{align}

Now, let $L^{2}(\Hq;e^{-|q|^2}d\lambda)$ denote the right Hilbert space of all quaternionic-valued square integrable functions on $\Hq$ with respect to
the inner product
\begin{align}\label{SP-full}
\scal{ f, g} =\int_{\Hq}\overline{f(q)}g(q) e^{-|q|^{2}}d\lambda(q) .
\end{align}
$d\lambda(q)= dx_0dx_1dx_2dx_3$ being the Lebesgue measure on $\Hq\equiv \R^4$.
The concrete description of $L^{2}(\Hq;e^{-|q|^2}d\lambda)$ is shown (\cite{El Hamyani}) to be given through the quaternionic Hermite polynomials defined by
\begin{equation}\label{20}
H_{m,n}(q,\overline{q})=m!n! \sum_{j=0}^{\min(m,n)}\dfrac{(-1)^{j}}{j!} \frac{q^{m-j}\overline{q}^{n-j}}{(m-j)!(n-j)!}.
\end{equation}
 They can be seen as natural extension of the complex Hermite polynomials. Their Rodriguez' formula involves the slice derivative
 \begin{equation}\label{SliceDer}
 \dfrac{\partial f}{\partial q}(q) = \partial_{I_q} f(x+I_qy) :=
\dfrac{1}{2}\left(\frac{\partial }{\partial x} - I_q\frac{\partial }{\partial y}\right)f_{I_q}(x+yI_q)
\end{equation}
and its quaternionic conjugate
 \begin{equation}\label{SliceDerConj}
  \dfrac{\partial f}{\partial \overline{q}}(q) = \overline{\partial_{I_q}} f(x+I_qy) :=
\dfrac{1}{2}\left(\frac{\partial }{\partial x} + I_q\frac{\partial }{\partial y}\right)f_{I_q}(x+yI_q)
\end{equation}
The basic and needed properties of $H_{m,n}(q,\overline{q})$ are summarized in the following items

\begin{itemize}

\item  The expression of the quaternionic Hermite polynomials can be written in terms of the confluent hypergeometric function ${_1F_1}$ as follows \begin{equation}\label{16}
H_{m,n}(q,\overline{q})= c_{m,n} r^{|m-n|}e^{(m-n)I\theta}
                                        {_1F_1}\left( \begin{array}{c} -\min(m,n) \\ |m-n|+1 \end{array}\bigg | r^{2} \right)
\end{equation}
where $q=re^{I\theta}$ with $r=|q| > 0$, $\theta \in [0,2\pi]$ and $I\in \Sq $, and
$$ c_{m,n} := \dfrac{(-1)^{\min(m,n)}\max(m,n)!}{|m-n|!}.$$

\item An exponential representation is given by
\begin{equation}\label{15}
H_{m,n}(q,\overline{q})=\exp\left(-\dfrac{\partial^{2}}{\partial q\partial \overline{q}}\right)(q^{m}\overline{q}^{n})= \exp\left(-\frac{1}{4}\left(\dfrac{\partial^{2}}{\partial x^2}+\dfrac{\partial^{2}}{\partial y^2}\right)\right)(q^{m}\overline{q}^{n})
\end{equation}

\item The quaternionic Hermite polynomials $H_{m,n}(q,\overline{q})$ belong to $L^{2}(\Hq;e^{-|q|^2}d\lambda)$ with square norm given by $\norm{H_{m,n}}^{2}_{L^{2}(\Hq;e^{-|q|^2}d\lambda)}=\pi m!n!$.
    Moreover, they form a complete orthogonal system. 
    We should point out that the expansion in terms of these polynomials is sliced (see Remark 3.4 in  \cite{El Hamyani}).

 \item We have the following bilateral generating function involving both the real and quaternionic Hermite polynomials
 \begin{equation}\label{gf}
 \sum_{n=0}^{+\infty}\dfrac{H_{n}(x)H_{m,n}(q,\bq)}{n!} =e^{-\bq^{2}+2x\bq} H_{m}\left(\bq+\frac{q}{2}-x\right).
 \end{equation}
\end{itemize}
Proofs and further properties of such polynomials can be found in \cite{El Hamyani}. 


\section{Discussion of the problematic}

In order to give a concrete description of the $\mathcal{C}^\infty$ and $L^{2}$-spectral analysis of the operator $\Delta_{q}$, we have to surmount two problems.
  The first one is connected to the uniqueness problem of the polar representation $q=re^{I\theta}$ and the slice representation $q=x+I y$, of given $q\in \Hq$.
  This can be removed by restricting $q$ to belong to specific subspaces of $\Hq$.
  Another problem that arise lies in the definition of the slice derivative given by \eqref{12}.
  Indeed, the operator $\Delta_{q}$ in \eqref{Laplaceian} takes the form
\begin{equation}\label{LaprealcoordH1}
\Box_{q} = -\frac{1}{4} \left( \frac{\partial^2}{\partial x^2} + \frac{\partial^2}{\partial y^2} \right) + \frac{1}{2} \left(x\frac{\partial}{\partial x} + y\frac{\partial}{\partial y} \right)  + \frac{I_q}{2} \left( x\frac{\partial}{\partial y} - y\frac{\partial}{\partial x} \right),
\end{equation}
on $\Hqr=\Hq\setminus \R$ and the form
\begin{equation}\label{LaprealcoordH2}
\Box_{x} = - \frac{\partial^2}{\partial x^2} + x\frac{\partial}{\partial x}
\end{equation}
on $\R$. A unified explicit form of the slice derivative is the following
\begin{equation}\label{SliceDerUnified}
\frac{\partial}{\partial q} = \frac{1}{2}  \left(\left(1+\chi_{\R}(q)\right)\frac{\partial}{\partial x} - \left(1-\chi_{\R}(q)\right) I_q \frac{\partial}{\partial y}\right)
\end{equation}
and therefore of the Laplacian $\Delta_{q}$ reads
\begin{align}\label{LaprealcoordHUnified}
\Delta_{q} = -\frac{1}{4}& \left\{\left(1+\chi_{\R}(q)\right)^2 \frac{\partial^2}{\partial x^2}
             + \left(1-\chi_{\R}(q)\right)^2 \frac{\partial^2}{\partial y^2} \right\}
            \\& +\frac{1}{2} \left(1+\chi_{\R}(q)\right) \left(x\frac{\partial}{\partial x} + y\frac{\partial}{\partial y} \right)
             +\frac{I_q}{2}  \left(1-\chi_{\R}(q)\right)\left( x\frac{\partial}{\partial y} - y\frac{\partial}{\partial x} \right). \nonumber
\end{align}
Accordingly, the operator $\Delta_{q}$ seen as second order differential operator on $\R^2$, is not elliptic nor uniform elliptic.
  However, it is semi-elliptic. 
  Indeed, the eigenvalues of the matrix
  $$  -\frac{1}{4} \left( \begin{array}{cc} \left(1+\chi_{\R}(q)\right)^2  & 0 \\ 0 & \left(1-\chi_{\R}(q)\right)^2  \end{array}\right)  ,$$
  associated to $\Delta_{q}$, are 
  clearly non-negatives (but not necessary positives).
Therefore, to provide a concrete description of the $L^{2}$-spectral analysis of the operator $\Delta_{q}$, we have to distinguish cases and consider special subspaces.

The strategy we will follow is to begin by studying the eigenvalue problem of $\Box_{q}$ in \eqref{LaprealcoordH1}
on $\Hqr$ when acting on both the $\mathcal{C}^\infty$ and $L^2$ quaternionic-valued functions on $\Hqr$ and next extend, in some how, the result to the whole $\Hq$. In fact, this will motivate the definition we will give to the generalized quaternionic Bargmann spaces that will generalize the slice hyperholomorphic Bargmann-Fock space \eqref{SHBFspace}.
It should be noted here that the Borel measurable set $\R$ is a negligeable set with respect to the gaussian measure on $\Hq$, and therefore
 \begin{align}\label{integration}
 \int_{\Hq} f(q) e^{-|q|^2}d\lambda(q) & = \int_{\Hqr} f(q) e^{-|q|^2}d\lambda(q) \nonumber
  \\& =  \int_{\R^{+*}\times ]0,2\pi[\times \Sq} f(re^{I\theta}) e^{-r^2} rdrd\theta d\sigma(I_q),
\end{align}
where $dr$ (resp. $d\theta$) denotes the Lebesgue measure on positive real line (the unit circle) and $d\sigma(I)$ stands for the standard area element on $\Sq$. This observation will be used systematically  in particular to obtain such extension
to the whole $\Hq$.


\section{$\mathcal{C}^\infty$ and $L^{2}$-concrete spectral analysis of the operator $\Box_{q}$ on $\Hqr$}
In this section, we consider the operator $\Box_{q}$ given through \eqref{LaprealcoordH1}.
Its expression in the polar coordinates $q=re^{I\theta}$, with $r>0$, $0\leq \theta \leq 2\pi$ and $I\in \mathbb{S}$, is given by
 the following

 \begin{lemma}\label{polar-action}
We have 
\begin{equation}\label{1}
\Box_{q}=-\dfrac{1}{4}\left(\dfrac{\partial^{2}}{\partial r^{2}}+\left[\dfrac{1}{r}-2r\right]\dfrac{\partial}{\partial r}+\dfrac{\partial^{2}}{r^{2} \partial \theta^{2}}-2I\dfrac{\partial}{\partial \theta}\right).
\end{equation}
Moreover, its action on the functions $e^{I n\theta}  a^{I}_{n}(r)$ is given by
\begin{equation}\label{action}
\Box_{q} e^{I n\theta}  a^{I}_{n}(r)=-\dfrac{e^{I n\theta}}{4r^2} \left[r^2\dfrac{\partial^{2}}{\partial r^{2}}+ (1-2r^2)r \dfrac{\partial}{\partial r}
+ (2nr^2 -n^2)\right]   a^{I}_{n}(r).
\end{equation}
 \end{lemma}

  \begin{proof}
  The formulas \eqref{1} and \eqref{action} follow by direct computations.
 \end{proof}

Now, let $\mu$ be a fixed quaternionic number and $\mathcal{E}^\infty_\mu(\Hqr,\Box_{q})$ be the corresponding $\mathcal{C}^\infty$-eigenspace
associated to the sliced differential operator $\Box_{q}$ defined in \eqref{Laplaceian}. It consists of all quaternionic-valued function that are $\mathcal{C}^\infty$ satisfying $\Box_{q}f= f \mu$ on $\Hqr$, to wit
 \begin{align}\label{smootheigenspace}
 \mathcal{E}^\infty_\mu(\Hqr,\Box_{q}) := \left\{ f\in \mathcal{C}^\infty(\Hqr); \,  \Box_{q}f= f \mu\right\}.
 \end{align}

The first main result of this section concerns the explicit characterization of the elements of $\mathcal{E}^\infty_\mu(\Hqr,\Box_{q})$.
Such description involves the left-confluent hypergeometric function ${_1F^{^L}_1}$ defined here for given $a,\xi \in \Hq$ and $c\in \R$ by  
 \begin{align}\label{NewHypergeometricFct}
 {_1F^{^L}_1}\left( \xi  \bigg | \begin{array}{c} a  \\ c \end{array} \right) = \sum_{n=0}^\infty\frac{\xi^n}{n!}  \frac{(a)_n}{(c)_n} ,
\end{align}
where $(a)_k$ denotes the Pochhammer symbol $(a)_k= a(a+1) \cdots (a+k-1)$
 with $(a)_0=1$.
Namely, we have

\begin{theorem}\label{thm:CInftydescription}
A quaternionic-valued function $f$ belongs to $\mathcal{E}^\infty_\mu(\Hqr,\Box_{q})$  if and only if it can be expanded in $\mathcal{C}^{\infty}(\Hqr)$  as
\begin{align}\label{leftEigenCinfty}
f(q)=\sum_{n\in \Z} q^{n}  \alpha^{I}_{n}   {_1F_1}\left( |q|^{2} \bigg | \begin{array}{c} -\mu \\ n+1 \end{array}\right) \beta^{I}_{\mu,n}.
\end{align}
for some quaternionic constants $\alpha^{I}_{n}\in \Hq$ and $\beta^{I}_{\mu,n}$ in the slice containing $\mu$, $\C_\mu$.
\end{theorem}

\begin{proof}
By smooth regularity, any $f\in\mathcal{C}^{\infty}(\Hqr)$ can be expanded as 
\begin{equation}\label{3}
f(re^{I\theta})=\sum_{n\in \Z}e^{I n\theta} a^{I}_{n}(r).
\end{equation}
The functions $(r,I) \longmapsto a^{I}_{n}(r)$ are $\mathcal{C}^{\infty}$ on $[0,+\infty[\times \Sq$.
Therefore, in view of Lemma \ref{polar-action} the associated right-eigenvalue problem $\Box_{q} f = f \mu$ reads
\begin{equation*}
-\sum_{n\in\Z} \dfrac{e^{I n\theta}}{4r^2} \left[r^2\dfrac{\partial^{2}}{\partial r^{2}}+ (1-2r^2)r \dfrac{\partial}{\partial r}
+ (2nr^2 -n^2)\right]  a^{I}_{n}(r) =    \sum_{n\in\Z} e^{I n\theta} a^{I}_{n}(r)  \mu  .
\end{equation*}
Identification of power series in $e^{I\theta} \in \C_I$, for fixed $r$ and $I$, yields
\begin{equation}\label{1a}
 \dfrac{1}{4r^2} \left[r^2\dfrac{\partial^{2}}{\partial r^{2}}+ (1-2r^2)r \dfrac{\partial}{\partial r}
+ (2nr^2 -n^2)\right]   a^{I}_{n}(r) =    -  a^{I}_{n}(r)  \mu
\end{equation}
for every $n$.
The change of variable $t=r^2$ and the change of function $a^{I}_{n}(r) = t^\alpha h_{n}(t,I)$ reduce \eqref{1a} to the following
\begin{equation}\label{1b}
t h_{n}^{''}(\cdot,I)  + (2\alpha +1 -t) h_{n}^{'}(\cdot,I)   + \dfrac{1}{4t} (2\alpha-n)(2\alpha+n -t)) h_{n}(\cdot,I) = - h_{n}(\cdot,I) \mu .
\end{equation}
The ansatz $\alpha=n/2$ shows that the $h_{n}(\cdot,I) $ satisfies the left-confluent hypergeometric differential equation:
\begin{equation}\label{6}
t h_{n}^{''}(\cdot,I) + (n +1 -t) h_{n}^{'}(\cdot,I)    = - h_{n}(\cdot,I) \mu .
\end{equation}
 The regular solution at $t=0$ of \eqref{6} is given by the left-confluent hypergeometric function ${_1F^{^L}_1}$ defined by \eqref{NewHypergeometricFct}. %
 Thus, we have
 $$ h_{n}(t,I)= \alpha^{I}_{n} \,  {_1F^{^L}_1}\left(  t \bigg | \begin{array}{c} -\mu \\ n+1 \end{array} \right) \beta^{I}_{\mu,n}$$ for
 some quaternionic constants $\alpha^{I}_{n} \in \Hq$ and $\beta^{I}_{\mu,n}$ in the slice containing $\mu$.
 Therefore,
\begin{align*}
f(re^{I\theta})=\sum_{n\in \Z} r^{n}e^{In\theta}  \alpha^{I}_{n} \,  {_1F^{^L}_1}\left( r^{2}\bigg | \begin{array}{c} -\mu \\ n+1 \end{array} \right)  \beta^{I}_{\mu,n}.
\end{align*}
\end{proof}

\begin{remark}
The regular solution of the left-confluent hypergeometric differential equation \eqref{6}, at $0$, is given by the left-confluent hypergeometric function ${_1F^{^L}_1}$ defined by \eqref{NewHypergeometricFct}
The consideration of ${_1F^{^L}_1}$ was essential for the lack of commutativity in quaternions. However, it coincides with the usual definiton of the confluent hypergeometric function ${_1F_1}$ (with ordered quaternionic parameters) for the variable being real.
\end{remark}

Added to the $\mathcal{C}^\infty$-version of the right-eigenvalue problem $\Box_{q}f= f \mu$ described by the previous theorem, one can also consider
the $L^2$-version. We denote by $\mathcal{F}_{\mu}^{2}(\Hqr)$ the $L^2$-eigenspace defined by
\begin{equation}\label{muL2eigen}
\mathcal{F}_{\mu}^{2}(\Hqr) :=  \left\{ f\in L^{2}(\Hqr;e^{-|q|^{2}}d\lambda); \,  \Box_{q} f = f \mu \right\}.
\end{equation}
In order to give a concrete description of such $L^2$-eigenspaces, we need first to establish some fundamental lemmas.
The first one shows that the considered space  $\mathcal{F}_{\mu}^{2}(\Hqr)$ can also be seen as a $L^2$-subspace of the $\mathcal{C}^\infty$-eigenspace $\mathcal{E}^\infty_\mu(\Hqr,\Box_{q})$. Namely, we assert

\begin{lemma}
We have
\begin{equation}\label{mu}
\mathcal{F}_{\mu}^{2}(\Hqr) :=  L^{2}(\Hqr;e^{-|q|^{2}}d\lambda) \cap \mathcal{E}^\infty_\mu(\Hqr,\Box_{q}).
\end{equation}
\end{lemma}

\begin{proof} This is an immediate consequence of the ellipticity of $\Box_{q}$ given by \eqref{LaprealcoordH1} and seen as a second order differential operator on $\R\times \R^*$.
\end{proof}

The second key lemma concerns the elementary functions
\begin{align}\label{elementaryFctmucst}
\varphi_{\mu,n}(\alpha^{I}_{n};q;\beta^{I}_{\mu,n}) := q^{n}  \alpha^{I}_{n} \,  {_1F_1}\left( \begin{array}{c} -\mu \\ n+1 \end{array}  \bigg | |q|^{2}  \right) \beta^{I}_{\mu,n}
\end{align}
for varying $n\in \Z$, where $q=x+Iy\in \Hqr$, $\alpha^{I}_{n} \in \Hq$ and $\beta^{I}_{\mu,n}$ in the slice containing $\mu$. 

\begin{lemma} \label{keyLem} We assert the following
\begin{enumerate}
         \item [(i)] The functions $\varphi_{\mu,n}$ are pairwisely orthogonal in the sense that
$ \scal{\varphi_{\mu,n},\varphi_{\mu,k}}  = 0$ whenever $n\ne k$.
         \item [(ii)]  The functions $\varphi_{\mu,n}$ belong to $L^{2}(\Hqr;e^{-|q|^{2}}d\lambda)$ if and only if $\mu$ is a nonnegative integer $\mu=m$ and $n\geq m$.
         \item [(iii)] Let $m=0,1,2,\cdots$. Then, the square norm of $\varphi_{m,n}$ in $L^{2}(\Hqr;e^{-|q|^{2}}d\lambda)$ is given by
\begin{align}\label{normelementaryFctmucst}
\norm{\varphi_{m,n}}^2_{L^{2}(\Hqr;e^{-|q|^{2}}d\lambda)} = \pi \dfrac{m!(n!)^2}{(m+n)!}  \int_{\Sq} | \alpha^{I}_{n} \beta^{I}_{\mu,n}|^2 d\sigma(I).
\end{align}
       \end{enumerate}
\end{lemma}

\begin{proof}
The first assertion follows by direct computation using the polar coordinates $q=re^{I\theta}$.
Indeed, in these coordinates, the Lebesgue measure $d\lambda$ becomes the product of the standard
Lebesgue measures $rdr$ on $\R^{+}$ and $d\theta$ on the unit circle times the standard area element $d\sigma(I)$ on $\Sq$,
the two-dimensional sphere of imaginary units in $\Hqr$. Therefore, by the Fubini's theorem, we have
\begin{align}
 \scal{\varphi_{\mu,n},\varphi_{\mu,k}} & =  \int_{\Hqr}  \overline{\varphi_{\mu,n}(\alpha^{I_q}_{n};q;\beta^{I_q}_{\mu,n})}
\varphi_{\mu,k}(\alpha^{I_q}_{k};q;\beta^{I_q}_{\mu,k}) e^{-|q|^2} d\lambda(q) \nonumber\\
 &=\int_{0}^\infty r^{n+k+1} \int_{\Sq} \overline{  \beta^{I}_{\mu,n} } R_{n,k}(I)  \beta^{I}_{\mu,k}  e^{-r^{2}}  d\sigma(I) dr ,
 \end{align}
 where $R_{n,k}(I)$ stands for
 $$ R_{n,k}(I) :=   \overline{  {_1F_1}\left(  \begin{array}{c} -\mu \\ n+1 \end{array} \bigg | r^{2} \right) } \overline{  \alpha^{I}_{n} }
                          \left(\int_{0}^{2\pi} e^{I(k-n)\theta} d\theta \right)  \alpha^{I}_{k}   {_1F_1}\left( \begin{array}{c} -\mu \\ k+1 \end{array}  \bigg | r^{2} \right).$$
Using the well-known fact $\int_{0}^{2\pi} e^{I(n-k)\theta} d\theta=2\pi \delta_{n,k}$ and making the change of variable $t=r^{2}$, we obtain
\begin{align}
 \scal{\varphi_{\mu,n},\varphi_{\mu,k}} &=  \pi  \left(\int_{\Sq} | \alpha^{I}_{n} \beta^{I}_{\mu,n}|^2 d\sigma(I)\right) \left(\int_{0}^\infty t^{n} \left|  {_1F_1}\left(  \begin{array}{c} -\mu \\ n+1 \end{array} \bigg | t \right) \right|^2  e^{-t}  dt \right)\delta_{n,k}. \label{elemIntegral}
 \end{align}
To prove the second assertion, we make use of the asymptotic behavior of the confluent hypergeometric function
$$ {_1F_1}\left( \begin{array}{c} a\\ c \end{array}  \bigg | t \right)  \sim  \frac{e^{t}t^{a-c}}{\Gamma(a)} $$
for $t$ large enough and $a \ne 0,-1,-2, \cdots$,
that follows from the Poincar\'e-type expansion \cite[Section 7.2]{Temme1996} 
$${_1F_1}\left( \begin{array}{c} a\\ c \end{array}  \bigg | t \right)  \sim  \frac{e^{t}t^{a-c}}{\Gamma\left(a\right)}%
\sum_{k=0}^{\infty}\frac{{\left(1-a\right)_{k}}{\left(c-a\right)_{k}}}{k!}t^{-k}.$$
Indeed, if $\mu\ne 0,1,2,\cdots$, then the nature of the integral involved in the right hand-side of \eqref{elemIntegral} is equivalent to
$$  \frac{1}{|\Gamma(-\mu)|^2} \int_{0}^\infty   t^{-(2\Re(\mu)+n+2)} e^{t}   dt$$ which is clearly divergent for $n$ large enough.
Conversely, if $\mu=0,1,2,\cdots$, the involved confluent hypergeometric function is the generalized Laguerre polynomial
 (\cite[Eq. (1), p. 200]{Rainville71})
\begin{align}\label{HyperLaguerre}
{_1F_1}\left( \begin{array}{c} -m \\ n+1 \end{array}\bigg | t \right)&=\dfrac{m!}{(n+1)_{m}}L_{m}^{n}(t)
\end{align}
which satisfies the orthogonality property \cite[Eq. (4), p. 205 - Eq. (7), p. 206]{Rainville71}
 \begin{align}\label{OrthLaguerre}
\int_{\mathbb{R}^{+}}L^{(\alpha)}_{j}(t)L^{(\alpha)}_{k}(t) t^{\alpha }e^{-t} dt
= \frac{\Gamma(\alpha +j+1)}{\Gamma(j+1) } \delta_{j,k}.
\end{align}
More precisely, starting from \eqref{elemIntegral}, the explicit computation yields
\begin{align*}
\norm{\varphi_{\mu,n}}^2_{L^{2}(\Hqr;e^{-|q|^{2}}d\lambda)} & =
 \pi \left( \dfrac{m!}{(n+1)_{m}}\right)^2 \left(\int_{0}^\infty (L_{m}^{n}(t))^2  t^{n}  e^{-t}  dt\right) \times \left(\int_{\Sq} | \alpha^{I}_{n} \beta^{I}_{\mu,n}|^2 d\sigma(I)\right)
 \\&
  = \pi \dfrac{m!(n!)^2}{(m+n)!}  \left( \int_{\Sq} | \alpha^{I}_{n} \beta^{I}_{\mu,n}|^2 d\sigma(I)\right) ,
\end{align*}
provided that $n+m+1>0$.
This completes the proof of assertions (ii) and (iii).
\end{proof}

\begin{remark}
For the particular case of $\alpha^{I}_{n}=1= \beta^{I}_{\mu,n}$, we denote the functions 
 in \eqref{elementaryFctmucst} simply
\begin{align}\label{elementaryFctmu}
\psi_{m,n}(q) := q^{n}   \,  {_1F_1}\left( \begin{array}{c} -m \\ n+1 \end{array}  \bigg | |q|^{2}  \right).
\end{align}
They satisfy the assertions of Lemma \ref{keyLem} above and their square norm read
\begin{align} \label{normeElemFct}
          \norm{\psi_{m,n}}^2_{L^{2}(\Hqr;e^{-|q|^{2}}d\lambda)}  = \pi \dfrac{m!(n!)^2}{(m+n)!}  Area(\Sq) .
\end{align}
\end{remark}

The second main result of this section is the following. It shows that the spectrum of $\Box_{q}$ acting  $L^{2}(\Hqr;e^{-|q|^{2}}d\lambda)$ is purely discrete and reduces to the quantized eigenvalues known as Landau levels.

\begin{theorem}\label{thm:description}
The space $\mathcal{F}_{\mu}^{2}(\Hqr)$ is nontrivial if and only if $\mu=m=0,1,2, \cdots.$ In this case, a nnonzero quaternionic-valued function $f$ belongs to $\mathcal{F}_{m}^{2}(\Hqr)$ if and only if it can be expanded as 
\begin{align}\label{expansionmonomials}
f(q)=\sum_{n=-m}^{+\infty} q^{n}   \,  {_1F_1}\left( \begin{array}{c} -m \\ n+1 \end{array}  \bigg | |q|^{2}  \right)  \Cn.
\end{align}
where the quaternionic constants $\Cn$ satisfy the growth condition
\begin{align}\label{Growth}
\norm{f}^{2}_{L^{2}(\Hqr;e^{-|q|^{2}}d\lambda)}  = \pi \sum_{n=-m}^{+\infty}
                \frac{m!(n!)^2}{(m+n)!} \left(\int_{\Sq} |\Cn|^2  d\sigma(I)\right) <+\infty,
\end{align}
\end{theorem}

\begin{proof}
Fix $\mu \in \Hq$ and assume that there is a nonzero function $f\in L^{2}(\Hqr;e^{-|q|^{2}}d\lambda)$ solution of $\Box_{q}f=\mu f$.
Then, the realization \eqref{mu} and Theorem \ref{thm:CInftydescription} show that $f$ admits the expansion
\begin{align*}
f(q)=\sum_{n\in\Z} q^{n}  \alpha^{I}_{n}  {_1F_1}\left( \begin{array}{c} -\mu \\ n+1 \end{array} \bigg | |q|^{2} \right)  \beta^{I}_{\mu,n}.
=\sum_{n\in\Z} \varphi_{\mu,n}(q) .
\end{align*}
Its square norm in the Hilbert space $L^{2}(\Hqr;e^{-\mid q\mid^{2}}d\lambda)$ can be computed using Lemma \ref{keyLem}. Indeed, the orthogonality of the
$(\varphi_{\mu,n})_n$ infers
\begin{align*}
\norm{f}^{2}_{L^{2}(\Hqr;e^{-|q|^{2}}d\lambda)} &    = \sum_{n\in\Z}   \norm{\varphi_{\mu,n}}^2_{L^{2}(\Hqr;e^{-|q|^{2}}d\lambda)}
 \end{align*}
 and therefore we have necessarily $\norm{\varphi_{\mu,n}}^2 $ is finite for every $n$, since $f$ belongs to $\mathcal{F}_{\mu}^{2}(\Hqr)$.
 In particular, we have $\norm{\varphi_{\mu,n_{0}}}^2 $ for some $n_{0}$ such that the integral mean
 $$ \int_{\Sq} |\alpha^{I}_{n_{0}} \beta^{I}_{\mu,n_{0}}|^2  d\sigma(I) \neq 0.$$
Such $n_{0}$ exists for $f$ being nonzero.
 This implies that $\mu$ is necessary of the form $\mu=m=0,1,2, \cdots$ with $n\geq -m$, which follows readily by means of (ii) in  Lemma \ref{keyLem}.
 In this case, the $\beta^{I}_{\mu,n}$ are reals (for $\mu=m\in \R$) and moreover we have
 \begin{align*}
\norm{f}^{2}_{L^{2}(\Hqr;e^{-|q|^{2}}d\lambda)} &    = \sum_{n=-m}^{+\infty}  \norm{\varphi_{m,n}}^2_{L^{2}(\Hqr;e^{-|q|^{2}}d\lambda)}
= \pi \sum_{n=0}^{+\infty}  \dfrac{m!(n!)^2}{(m+n)!}  \int_{\Sq} | \Cn |^2 d\sigma(I) ,
 \end{align*}
where we have set $\Cn:= \alpha^{I}_{n} \beta^{I}_{\mu,n}$. This yields the growth condition \eqref{Growth} and thus the proof is completed.
\end{proof}

According to the fact that the quaternionic Hermite polynomials $H_{m,n}(q,\overline{q})$ form a complete orthogonal system in $L^{2}(\Hqr;e^{-|q|^2}d\lambda)$
 (see \cite{El Hamyani}), an expansion of the elements of $\mathcal{F}_{m}^{2}(\Hqr)$ in terms of the $H_{m,n}(q,\overline{q})$ can be given.
  The following result describes such expansion.

\begin{corollary}\label{cor:expansionHermite}
 The space $\mathcal{F}_{m}^{2}(\Hqr)$ contains the quaternionic Hermite polynomials defined by \eqref{20}.
 Moreover, every element $f$ belonging to $\mathcal{F}_{m}^{2}(\Hqr)$ can be expanded as
\begin{align}\label{expansionHermite}
f(q) &=\sum_{n=-m}^{+\infty} \dfrac{(-1)^{m}n!}{(m+n)!} H_{m+n,m} (q,\overline{q}) \Cn
\end{align}
for some sliced quaternionic constants $\Cn$ displaying the growth condition \eqref{Growth}.
\end{corollary}

\begin{proof}
Making appeal of \eqref{16}, the confluent hypergeometric function involved in \eqref{expansionmonomials}
can be rewritten in terms of the quaternionic Hermite polynomials as
\begin{align}\label{HypHermite}
 q^{n} {_1F_1}\left( \begin{array}{c} -m \\ n+1 \end{array}  \bigg | |q|^{2}  \right) = \dfrac{(-1)^{m}n!}{(m+n)!}  H_{m+n,m} (q,\overline{q}).
\end{align}
Therefore, the expression of $f(q)$ given through \eqref{expansionmonomials} reduces further to \eqref{expansionHermite}
with the same growth condition \eqref{Growth}.
\end{proof}

We conclude this section by a result concerning the right quaternionic Hilbert space $\mathcal{F}^{2}_{full}(\Hqr)$, defined as the space of all
 slice regular functions that are $e^{-|q|^2}d\lambda$-square integrable on $\Hqr$,
\begin{align} \label{full}
\mathfrak{F}_{full}^{2}(\Hqr) :=\mathcal{SR}(\Hqr)\cap L^{2}(\Hqr;e^{-|q|^2}d\lambda) .
\end{align}
Namely, we assert the following

\begin{corollary} \label{cor:motiv1}
The right quaternionic Hilbert space $\mathcal{F}_{0}^{2}(\Hqr)$, corresponding to $m=0$, coincides with the full hyperholomorphic Bargmann-Fock space
$ \mathcal{F}^{2}_{full}(\Hqr)$ given by \eqref{full}.
\end{corollary}

\begin{proof}
This follows readily from Corollary \ref{cor:expansionHermite} combined with the fact $H_{n,0} (q,\overline{q})=q^n$.
Indeed, for the special case of $m=0$, we get
\begin{align}\label{SeqCharFull}
\mathcal{F}_{0}^{2}(\Hqr) = \left\{
f(q) = \sum_{n=0}^{+\infty} q^{n}  \Cn; \,
    \sum_{n=0}^{+\infty} n! \left(\int_{\Sq} |\Cn|^2  d\sigma(I)\right) <+\infty
\right\}.
\end{align}
This is exactly the sequential characterization of the full hyperholomorphic Bargmann-Fock space
$ \mathcal{F}^{2}_{full}(\Hqr)$. Indeed, for given slice regular functions
$$f(q) = \sum_{n=0}^\infty q^n a_n  \quad \mbox{and} \quad f(q) = \sum_{n=0}^\infty q^n b_n ,$$
for some quaternionic sliced constants $a_n$ and $b_n$, we have
\begin{align} \label{scalarproductH}
\scal{f,g}_{L^{2}(\Hqr;e^{-|q|^{2}}d\lambda)} = \pi \sum_{n=0}^{+\infty} n! \left(\int_{\Sq} \overline{a_n} b_n   d\sigma(I)\right).
\end{align}
Therefore, the norm boundedness of a given slice regular function $f(q) = \sum\limits_{n=0}^\infty q^n \Cn$ reads
$$ \norm{f}^2_{L^{2}(\Hqr;e^{-|q|^{2}}d\lambda)} = \pi\sum_{n=0}^{+\infty} n! \left(\int_{\Sq} |\Cn|^2  d\sigma(I)\right) < +\infty.$$
This completes the proof.
\end{proof}

\begin{corollary} \label{cor:motiv2}
We have
\begin{align}
\mathcal{F}^{2}_{slice}(\Hq)  \subset  \mathcal{F}_{0}^{2}(\Hqr) = \mathfrak{F}_{full}^{2}(\Hqr),
\end{align}
where $ \mathcal{F}^{2}_{slice}(\Hq)$ is the slice hyperholomorphic Bargmann-Fock space given by \eqref{SHBFspace}.
\end{corollary}

\begin{proof}
The inclusion follows immediately by comparing the sequential characterization of the full hyperholomorphic Bargmann-Fock space
$ \mathcal{F}^{2}_{full}(\Hqr)$ given through \eqref{SeqCharFull} and the one for the slice hyperholomorphic Bargmann-Fock space
$ \mathcal{F}^{2}_{slice}(\Hq)$ given by Proposition 3.11 in \cite{AlpayColomboSabadiniSalomon2014}, to wit
$$\mathcal{F}_{slice}^{2}(\Hq) = \left\{ f(q) =\sum_{n=0}^{+\infty} q^{n}  C_n; \,
   \norm{f}^{2}_{L^{2}(\C_I;e^{-|q|^{2}}d\lambda_I)} = \pi \sum_{n=0}^{+\infty} n!|C_{n}|^{2}<+\infty \right\} .$$
\end{proof}

\section{Generalized quaternionic Bargmann spaces and their reproducing kernels }

Motivated by Corollary \ref{cor:motiv2} and using the functions $\psi_{m,n}$ defined through \eqref{elementaryFctmu}, to wit
\begin{align*}
\psi_{m,n}(q) := q^{n}   \,  {_1F_1}\left( \begin{array}{c} -m \\ n+1 \end{array}  \bigg | |q|^{2}  \right),
\end{align*}
we introduce an appropriate class of infinite dimensional right quaternionic Hilbert spaces. They are subspaces of the $\mathcal{F}_{m}^{2}(\Hqr)$ and possessing reproducing kernels. In fact, for every fixed nonnegative integer $m$, we define $\mathcal{GB}_{m}^{2}(\Hq)$ to be the space spanned by the functions $\psi_{m,n}$ and equipped with the scaler product \eqref{SP-full}. That is, $\mathcal{GB}_{m}^{2}(\Hq)$ consists of the series
 \begin{align}\label{expansionmonomials}
\sum_{n=-m}^{+\infty} q^{n}  {_1F_1}\left( \begin{array}{c} -m \\ n+1 \end{array}  \bigg | |q|^{2}  \right)  C_n,
\end{align}
where the constants $C_n\in \Hq$ satisfy the growth condition
 \begin{align}\label{growthGQBS}
  \norm{f}^2:=\sum_{n=-m}^{+\infty} \frac{\pi m!(n!)^2}{(m+n)!} |C_n|^2   <+\infty .
  \end{align}
In view of \eqref{HypHermite}, we can suggest an equivalent definition of the $\mathcal{GB}_{m}^{2}(\Hq)$. Namely, we have
 \begin{align}\label{expansionmonomials1}
\mathcal{GB}_{m}^{2}(\Hq) := \left\{ f(q)=
\sum_{n=0}^{+\infty} H_{n,m}(q,\overline{q})  C_n; \, C_n\in \Hq \, \mbox{such that }  \pi m! \sum_{n=0}^{+\infty} k! |C_n|^2   <+\infty \right\}.
\end{align}
Clearly $ \mathcal{GB}_{0}^{2}(\Hq) $, corresponding to $m=0$, coincides with $\mathcal{F}^{2}_{slice}(\Hq)$ given through \eqref{SHBFspace}.

\begin{definition} \label{GBSm}
The space $\mathcal{GB}_{m}^{2}(\Hq)$, generalizing the slice hyperholomorphic Bargmann-Fock space $\mathcal{F}^{2}_{slice}(\Hq)$,
is called the generalized quaternionic Bargmann space of level $m$.
\end{definition}

Accordingly, it is not difficult to see that for every fixed $m$, the space $\mathcal{GB}_{m}^{2}(\Hq)$ in \eqref{expansionmonomials1} is a Hilbert subspace of $L^{2}(\Hq; e^{-|q|^2}d\lambda)$. Moreover, the quaternionic Hermite polynomials $H_{n,m}$, for varying $n=0,1,2,\cdots$, are generators of it. Their linear independence is equivalent to their completion. Thus, we can show that a given $f\in\mathcal{GB}_{m}^{2}(\Hq) $  is identically zero on $\Hq$ whenever
 $\scal{f,H_{n,m}}=0$ for every $n=0,1,2,\cdots$. 
This result is reformulated as follows

\begin{theorem}\label{AmHilbertBasis}
The spaces $\mathcal{GB}_{m}^{2}(\Hq) $ are right quaternionic Hilbert spaces. The quaternionic Hermite polynomials $H_{n,m}$, for varying $n=0,1,2,\cdots$, and fixed nonnegative integer $m$, belong to $\mathcal{GB}_{m}^{2}(\Hq) $ and constitute an orthogonal basis of it.
\end{theorem}

\begin{proof}
We begin by noting that the $H_{k,m}$ is an orthogonal system (with respect to the both indices) with respect to the gaussian measure.
Now, starting from the expansion $ f = \sum_{k=0}^{+\infty} H_{k,m}  C_k $ for $f\in\mathcal{GB}_{m}^{2}(\Hq) $, the direct computation yields
 \begin{align*}
 \scal{f,H_{n,m}}
 &= \int_{\Hq} \sum_{k=0}^{+\infty} \overline{C_k H_{k,m} (q,\overline{q})} H_{n,m}(q,\overline{q}) e^{-|q|^2}d\lambda(q)
 \\& = \lim\limits_{R \to +\infty} \int_{B(0,R)} \sum_{k=0}^{+\infty} \overline{C_k H_{k,m}(q,\overline{q})} H_{n,m}(q,\overline{q}) e^{-|q|^2}d\lambda(q)
 \\& = \lim\limits_{R \to +\infty} \sum_{k=0}^{+\infty}  \int_{B(0,R)} \overline{C_k H_{k,m}(q,\overline{q})}  H_{n,m}(q,\overline{q}) e^{-|q|^2}d\lambda(q).
 \end{align*}
 Hence, using the explicit expression of the quaternionic Hermite polynomials and integrating on $B(0,R)$ with respect to the polar coordinates, one shows that
  \begin{align*}
  \int_{B(0,R)} \overline{C_k H_{k,m}(q,\overline{q})}  H_{n,m}(q,\overline{q})  e^{-|q|^2}d\lambda(q)
= \overline{C_n} \left(\int_{B(0,R)}  \left| H_{n,m}(q,\overline{q}) \right|^2e^{-|q|^2}d\lambda(q)\right) \delta_{n,k} .
 \end{align*}
Subsequently, the expression of $\scal{f,H_{n,m}}$ becomes
\begin{align*}
 \scal{f,H_{n,m}} = \overline{C_n} \lim\limits_{R \to +\infty} \int_{B(0,R)}  \left|  H_{n,m}(q,\overline{q}) \right|^2e^{-|q|^2}d\lambda(q)
  = \overline{C_n} \norm{ H_{n,m}}^2 .
 \end{align*}
Therefore, $\overline{C_n}=0$ for every nonnegative integer $n$, by the assumption $\scal{f,H_{n,m}}=0$ for every $n$. This proves that $f\equiv 0$ on $\Hqr$.
\end{proof}

In the sequel, we establish further properties of the generalized quaternionic Bargmann-Fock space $\mathcal{GB}_{m}^{2}(\Hq)$.
The first one shows that $\mathcal{GB}_{m}^{2}(\Hq)$ is a reproducing kernel quaternionic Hilbert space. To this end, the following lemma is needed.


\begin{lemma}
For every fixed $q\in \Hq$, the evaluation map $\delta_{q}f=f(q)$ is a continuous linear form on the Hilbert space $\mathcal{GB}_{m}^{2}(\Hq)$.
\end{lemma}

\begin{proof}
Let $f\in \mathcal{GB}_{m}^{2}(\Hq)$ and expand it in $\mathcal{C}^{\infty}(\Hq)$ as $f(q)= \sum_{n=0}^{+\infty}  H_{n,m}(q,\overline{q}) C_{n}$.
 Thus, using the Cauchy-Schwartz inequality and the expression of the square norm $\norm{f}^{2}=\pi m! \sum_{n=0}^{+\infty} n! |C_{n}|^{2} $, we obtain
\begin{equation}\label{11}
|f(q)|  
\leq \left(\sum_{n=0}^{+\infty} \frac{| H_{n,m}(q,\overline{q}) |^2}{\pi m!n!}\right)^{\frac{1}{2}} \norm{f}_{m} .
\end{equation}
The series in the right hand-side of \eqref{11} is absolutely convergent for every fixed $r$ and is independent of $f$. This follows readily making use of
the following upper bound (see \cite[Corollary 4.3]{El Hamyani}):
\begin{align}\label{estimate}
\left| H_{n+k,n}(q,\bq )\right|\leq \dfrac{(n+k)!}{k!}\left| q\right|^{k}  e^{\frac{|q|^2}{2}}.
\end{align}
\end{proof}

\begin{remark} More explicitly, by means of \cite[Corollary 3.3]{BenahmadiGElkachkouri2017}, we have
\begin{align}\label{sumAbs}
\sum_{n=0}^{+\infty} \frac{| H_{n,m}(q,\overline{q}) |^2}{\pi m!n!} = \frac{ e^{ |q|^2 }}{\pi} .
\end{align}
\end{remark}

The next result gives the explicit expression of the reproducing kernel of the $\mathcal{GB}_{m}^{2}(\Hq)$,  which exists by means of the quaternionic version of the Riesz representation theorem combined with the previous Lemma.


\begin{theorem}\label{thm:RepKernel}
 The reproducing kernel of the generalized quaternionic Bargmann-Fock space $\mathcal{GB}_{m}^{2}(\Hq)$ is given by
\begin{align*}
\mathcal{K}_{m}(q,q')&= \dfrac{e_{*}^{[\overline{q},q']}}{\pi} L_m(|q - q'|^2),
\end{align*}
where
$$e_{*}^{[a,b]}:=\sum_{n=0}^{+\infty}\dfrac{a^{n}b^{n}}{n!}$$
and $L_m(x)$ stands for the classical Laguerre polynomial of degree $m$.
\end{theorem}

\begin{proof} Recall that $(H_{n,m})_n$ is a orthogonal basis of $\mathcal{GB}_{m}^{2}(\Hq)$ (see Theorem \ref{AmHilbertBasis}).
Thus, the computation of $\mathcal{K}_{m}(q,q')$ can be done by performing
$$\mathcal{K}_{m}(q,q')=\dfrac{1}{\pi m!}\sum_{n=0}^{+\infty}\dfrac{H_{n,m}(q,\overline{q})H_{m,n}(q',\overline{q'})}{n!}.$$
Notice for instance that for the particular case of $q$ and $q'$ belonging to the same slice, the result follows by means of
$$\sum\limits_{n=0}^{+\infty} \frac{ H_{m,n}(z,\bz ) H_{n,m}(w,\bw ) }{\pi m!n!  }    =   \frac{e^{\bz w}}{\pi m!}  H_{m,m}( z -w, \bz - \bw) =   \frac{e^{ \bz  w }}{\pi}  L_m(|z -w|^2) $$
which is readily an immediate consequence of Theorem 3.1 in \cite{BenahmadiGElkachkouri2017}, to wit
$$
 \sum\limits_{n=0}^{+\infty} \frac{ t^n }{n! \nu^n  }  H_{m,n}^{\nu}(z,\bz ) H_{n,m'}^{\nu}(w,\bw )  =
    t^{m'}   H_{m,m'}^{\nu}( z -tw, \bz - \overline{t}\bw) e^{\nu  t \bz w},
$$
valid for every $t$ in the unit circle and $z,w\in \C$, 
combined with $ H_{m,m}(\xi,\bar\xi) =  m! L_m(|\xi|^2)$. 
Therefore, we claim
$$ \mathcal{K}_{m}(q,q') =  \dfrac{e_{*}^{[\overline{q},q']}}{\pi} L_m(|q - q'|^2) .$$
\end{proof}

\begin{remark} The operator $f \longmapsto P_mf$ given by 
\begin{align}\label{OrthProj}
Pf(q) = \int_{\Hq} \mathcal{K}_{m}(q,q') f(q') e^{-|q'|^2}d\lambda(q') = \int_{\Hq}\dfrac{e_{*}^{[\overline{q},q']}}{\pi} L_m(|q - q'|^2) f(q') e^{-|q'|^2}d\lambda(q') 
\end{align}
define the orthogonal projection of $L^{2}(\Hq;e^{-|q|^{2}}d\lambda)$ to $\mathcal{GB}_{m}^{2}(\Hq)$.
\end{remark}

We conclude this section with the following result giving an orthogonal Hilbertian decomposition
 of the Hilbert space $L^{2}(\Hq;e^{-|q|^{2}}d\lambda)$.

\begin{theorem}\label{HilbDecomp}
We have the following hilbertian decomposition
\begin{align*}
L^{2}(\Hq;e^{-|q|^{2}}d\lambda)=\bigoplus_{m\geq 0}\mathcal{GB}_{m}^{2}(\Hq).
\end{align*}
\end{theorem}

\begin{proof}
Such decomposition is equivalent to prove that the orthogonal complement of $\bigoplus\limits_{m\geq 0} \mathcal{GB}_{m}^{2}(\Hq)$
in $L^{2}(\Hq;e^{-|q|^{2}}d\lambda)$ reduces to $\{0\}$. To this end, let $f\in \left(\bigoplus\limits_{m\geq 0} \mathcal{GB}_{m}^{2}(\Hq)\right)^\perp$.  Then, in particular we have
\begin{align*}
\int_{\Hq} e_{*}^{[\overline{q},w]} L_{m}(|w-q|^{2})\overline{f(w)}e^{-\mid w\mid^{2}}d\lambda(w)=0
\end{align*}
for every fixed $q\in \Hq$ and every $m=0,1,2,\cdots$.
Thus, for given $t\in ]0,1[$, we get
\begin{align*}
\sum_{m=0}^{N}\int_{\Hq} e_{*}^{[\overline{q},w]} t^{m} L_{m}(|w-q|^{2})\overline{f(w)}e^{-\mid w\mid^{2}}d\lambda(w)=0.
\end{align*}
By tending $N$ to $+\infty$ and using the explicit formula for the generating function of the Laguerre polynomials (\cite[Eq. (14), p. 135]{Rainville71})
\begin{align*}
\sum_{n=0}^\infty  \xi^{n} L^{(\alpha)}_{n}(t) = \frac{1}{(1-\xi)^{\alpha +1}} \exp\left(\frac{t\xi}{\xi - 1} \right),
\end{align*}
 we obtain
\begin{align*}
\int_{\Hq} e_{*}^{[\overline{q},w]} \dfrac{e^{-\dfrac{t|q-w|^{2}}{1-t}}}{1-t}\overline{f(w)}e^{-\mid w\mid^{2}}d\lambda(w)=0.
\end{align*}
The limit $t\longrightarrow 1^{-}$ yields an integral involving the Dirac $\delta$-function at the point $q\in \Hq$. 
From that we deduce $e_{*}^{[\overline{q},w]} \overline{f(w)}e^{-\mid w\mid^{2}}$ and therefore $\overline{f(q)}=0$ for every $q\in \Hq$.
\end{proof}

\begin{remark} 
Theorem \ref{HilbDecomp} is contained in \cite[Theorem 3.3]{El Hamyani}, since the $H_{n,m}$ is a basis of $L^{2}(\Hq;e^{-|q|^{2}}d\lambda)$.
But here we have provide a different proof based on the explicit closed formula for the reproducing kernel of the spaces $\mathcal{GB}_{m}^{2}(\Hq)$.
\end{remark} 

\section{Generalized quaternionic Bargmann transforms $\mathcal{B}_m$}

In this section, we introduce a family of generalized quaternionic Segal-Bargmann transforms defined on the quaternionic Hilbert space $L^{2}_{\Hq}(\R;dt)$, consisting of all square integrable $\Hq$-valued functions with respect to the inner product
\begin{align*}
\scal{ f,g }_{ L^{2}(\R;\Hq)}:=\int_{\R}f(t) \overline{g(t)}  dt.
\end{align*}
Their images will be the generalized quaternionic Bargmann-Fock spaces defined and studied in the previous section.
To this end, we define the kernel function $ A(x;q)$ on $\R\times \Hq$ to be the bilinear generating function of the real Hermite functions,
\begin{align}\label{6}
h_{n}(t)=(-1)^{n}e^{\frac{t^{2}}{2}}\dfrac{d^{n}}{dt^{n}}(e^{-t^{2}}),
\end{align}
that form an orthogonal basis of $L^2_{\Hq}(\R;dt)$, with norm
\begin{align}\label{1}
\|h_{n}\|^{2}_{L^2_{\Hq}(\R;dt)}=2^{n}n!\sqrt{\pi},
\end{align}
and the quaternionic Hermite polynomials $H_{m,n}(q,\overline{q})$, which form an orthogonal basis of $L^{2}(\Hq;e^{-|q|^{2}}d\lambda)$, with norm
\begin{align}\label{2}
\|H_{m,n}\|^{2}_{L^{2}(\Hq,e^{-|q|^{2}}d\lambda)}=\pi m!n!.
\end{align}
That is
\begin{align}\label{kernel}
 A(x;q) &= \sum_{n=0}^{+\infty}\dfrac{h_{n}(t)H_{m,n}(q,\overline{q})}{\|h_{n}\|\|H_{m,n}\|} .
\end{align}
Thus, we assert

\begin{theorem} \label{thm:KernelFctSBTm}
For every $t\in \R$ and $q \in \Hq$, we have
\begin{align*}
A_{m}(t;q)&=\frac{\exp\left(-\frac{t^{2}}{2}-\frac{\overline{q}^{2}}{2}+\sqrt{2}\overline{q}t\right)}{(\pi)^{\frac{3}{4}}(\sqrt{2})^{m}\sqrt{m!}}
 H_{m}\left(\frac{q+\overline{q}}{\sqrt{2}}-t\right).
\end{align*}
\end{theorem}

\begin{proof}
By means of the explicit expressions of the norms of $h_{n}$ (see \eqref{1}) and $H_{m,n}$ (see \eqref{2}), and making use of the fact that $H_{m,n}(q,\overline{q})=e^{-\Delta_{S}}(q^{m}\overline{q}^{n})$, we obtain
\begin{align*}
A_{m}(t;q)&=\dfrac{e^{-\frac{t^{2}}{2}}}{(\pi)^{\frac{3}{4}}\sqrt{m!}}
e^{-\Delta_{S}}\left(q^{m}\sum_{n=0}^{+\infty}\frac{\overline{q}^{n}}{\sqrt{2}^{n}n!}H_{n}(t)\right)\\
&=\dfrac{e^{-\frac{t^{2}}{2}}}{(\pi)^{\frac{3}{4}}\sqrt{m!}}\sum_{j=0}^{m}\dfrac{(-1)^{j}m!q^{m-j}(\sqrt{2})^{-j}}{j!(m-j)!}
\left(\sum_{n=j}^{+\infty}\frac{\overline{q}^{n-j}}{\sqrt{2}^{n-j}(n-j)!}H_{n-j}(t)\right)\\
&=\dfrac{e^{-\frac{t^{2}}{2}}}{(\pi)^{\frac{3}{4}}\sqrt{m!}}\sum_{j=0}^{m}\dfrac{(-1)^{j}m!q^{m-j}(\sqrt{2})^{-j}}{j!(m-j)!}
\left(\sum_{k=0}^{+\infty}\frac{\overline{q}^{k}}{\sqrt{2}^{k} k!}H_{k+j}(t)\right).\\
\end{align*}
The last equality holds thanks to the change of indices $k=n-j$. Using the fact,
\begin{align*}
\sum_{k=0}^{+\infty}\frac{\overline{q}^{k}}{\sqrt{2}^{k}k!}H_{k+j}(t)= \exp\left(-\frac{\overline{q}^{2}}{2}+\sqrt{2}t\overline{q}\right)H_{j}\left(t-\dfrac{\overline{q}}{\sqrt{2}}\right),
\end{align*}
we obtain
\begin{align*}
A_{m}(t;q)&=\dfrac{\exp\left(-\frac{\overline{q}^{2}}{2}+\sqrt{2}t\overline{q}\right)}{(\pi)^{\frac{3}{4}}\sqrt{2}^{m}}
\sum_{j=0}^{m}\dfrac{m!(\sqrt{2}q)^{m-j}}{j!(m-j)!} H_{j}\left(\dfrac{\overline{q}}{\sqrt{2}}-t\right).
\end{align*}
Finally, the result follows by utilizing the fact that
\begin{align*}
\sum_{j=0}^{m}\binom{m}{j} H_{j}(t) (2\xi)^{m-j}&=H_{m}(t+\xi).
\end{align*}
\end{proof}

Associated to the kernel function $A(x;q)$ given through (\ref{kernel}), we consider the integral transform defined by
\begin{align*}
[\mathcal{B}_m\phi](q)&:= \int_{\R}A_{m}(t;q)\phi(t)dt
\\& = \left(\dfrac{1}{\pi}\right)^{\frac 34}  \frac{1}{(\sqrt{2^{m} m!}}
\int_{\R}e^{-\frac{t^{2}}{2}-\frac{\overline{q}^{2}}{2}+\sqrt{2}\overline{q}t}
H_{m}\left(\frac{q+\overline{q}}{\sqrt{2}}-x\right)\phi(t)dt;
\end{align*}
for a given function  $\phi:\R\rightarrow \Hq$, provided that the integral exists.
The following result shows that $\mathcal{B}_m$ is well-defined on $L^2_{\Hq}(\R;dt)$. Namely, we have

\begin{theorem} \label{thm:isometrySBTm}
For a fixed $q\in \Hq$, the function
\begin{align*}
A_{m;q}:t\longrightarrow A_{m}(t;q):= \left(\dfrac{1}{\pi}\right)^{\frac 34}  \frac{1}{\sqrt{2^{m} m!}}
e^{-\frac{t^{2}}{2}-\frac{\overline{q}^{2}}{2} +\sqrt{2}\overline{q}t}
H_{m}\left(\frac{q+\overline{q}}{\sqrt{2}}-t\right)
\end{align*}
belongs to $L^2_{\Hq}(\R;dt)$, and we have
\begin{align}\label{5}
\norm{ A_{m;q}}{L^2_{\Hq}(\R;dt)}=\dfrac{1}{\sqrt{\pi}}e^{\frac{|q|^{2}}{2}}.
\end{align}
Moreover, for every quaternion $q\in \Hq$ and every $\phi \in L^2_{\Hq}(\R;dt) $, we have
\begin{align*}
|[\mathcal{B}_m\phi](q)| &\leq \dfrac{1}{\sqrt{\pi}}e^{\frac{q^{2}}{2}}\norm{ \phi }_{L^2_{\Hq}(\R;dt)}.
\end{align*}
\end{theorem}

\begin{proof}
fix $q=x+Iy$ in $\Hq$ and write the modulus of the kernel function $A_{m}(t;q)$ as
\begin{align*}
\left| A_{m}(t;q)\right|^{2}&=\left(\dfrac{1}{\pi}\right)^{\frac 32}  \frac{1}{2^{m}}
\left|e^{-\frac{t^{2}}{2}-\frac{x^{2}}{2}-\frac{y^{2}}{2}+Ixy+\sqrt{2}q_{1}t-I\sqrt{2}q_{2}t}\right|^{2}
\left|H_{m}(\sqrt{2}q_{1}-t)\right|^{2}\\
&=\left(\dfrac{1}{\pi}\right)^{\frac 32}  \frac{1}{2^{m}} e^{-t^{2}-x^{2}+y^{2}+\sqrt{2}xt} \left|H_{m}(\sqrt{2}x-t)\right|^{2}.
\end{align*}
Therefore, it follows that
\begin{align*}
\|A_{m;q}\|^{2}_{L^2_{\Hq}(\R;dt)}&=(\pi)^{\frac{-3}{2}}2^{-m}e^{x^{2}+y^{2}}\int_{\R}e^{-(t-\sqrt{2}q_{1})^{2}}|H_{m}(t-\sqrt{2}q_{1})|^{2}dt\\
&=(\pi)^{\frac{-3}{2}}2^{-m}e^{|q|^{2}}\int_{\R}e^{-u^{2}}\mid H_{m}(u)\mid^{2}du.
\end{align*}
Using the norm of the real Hermite polynomials which equal to $\sqrt{\pi}m!2^{m}$ we have
\begin{align*}
\norm{ A_{m;q}}{L^2_{\Hq}(\R;dt)}= \frac{1}{\sqrt{\pi}}e^{\frac{|q|^{2}}{2}}.
\end{align*}
Using the Cauchy-Schwartz inequality, we obtain
\begin{align}\label{4}
|\mathcal{B}_m\phi(q)| &\leq  \int_{\R}|A_{m}(t;q)|  |\phi(t)| dt\leq \norm{A_{m;q}}_{L^2_{\Hq}(\R;dt)} \norm{\phi}_{L^2_{\Hq}(\R;dt)}.
\end{align}
In view of \eqref{5} the inequality \eqref{4} reduces simply to
\begin{align*}
|\mathcal{B}_m\phi(q)| &\leq \dfrac{e^{ \frac{|q|^{2}}{2}}}{\sqrt{\pi}} \norm{ \phi }_{L^2_{\Hq}(\R;dt)}.
\end{align*}
\end{proof}

\begin{remark}
By comparing \eqref{5} and  \eqref{5} to \eqref{sumAbs}, we conclude that $\norm{A_{m;q}}_{L^2_{\Hq}(\R;dt)}= \sqrt{K_m(q,q)}$ for every $q\in\Hq$.
\end{remark}

\begin{remark}
 The Segal-Bargmann transform $\mathcal{B}_m$ maps the orthogonal basis of $L^2_{\Hq}(\R;dt)$ consisting of the Hermite polynomials $h_{n}$ to the orthogonal basis of the generalized Bargmann-Fock spaces consisting of the quaternionic Hermite polynomials. More exactly, we have 
 $$[\mathcal{B}_m(h_{n})](q)= \dfrac{(\sqrt{2})^{m-1}}{\pi}H_{m,n}(q,\overline{q}).$$
\end{remark}

\section{A left-sided uaternionic Fourier-Wigner transform}

We conclude this paper by introducing the quaternionic Fourier-Wigner transform. We give its action on the real Hermite polynomials and we establish its connection to the generalized quaternionic Segal-Bargmann transform and the Fourier-Wigner transform.

\begin{definition}
For fixed $I\in \mathbb{S}$ and for any $f,g\in L^2_{\Hq}(\R;dt),$ we define the left-sided quaternionic Fourier-Wigner transform as the quaternionic-valued function $V_{I}(f,g)$ on $\R\times\R$ given by:
\begin{align}
V_{I}(f,g)(x+Iy):=\dfrac{1}{\sqrt{2\pi}}\int_{\R}e^{Iyt}f\left(t+\frac{x}{2}\right)g\left(t-\frac{x}{2}\right)dt.
\end{align}
\end{definition}

The following result gives the explicit expression of the action of the Fourier-Wigner transform on the real Hermite polynomials in \eqref{6}. We assert

\begin{theorem}\label{8}
For fixed $I \in \Sq$ and every $x,y\in \R$, we have
\begin{align*}
V_{I_q}(h_{m},h_{n})(q)&=(-1)^{n}(\sqrt{2})^{m+n-1}e^{-\frac{|q|^{2}}{2}}H_{m,n}\left(\frac q{\sqrt{2}},\frac{\overline{q}}{\sqrt{2}}\right).
\end{align*}
\end{theorem}

The following Lemma is needed for proving the previous theorem. 

\begin{lemma}\label{1888}
For $\alpha>0$ and $\beta\in  \Hq$, we have
\begin{align*}
\int_{-\infty}^{+\infty} e^{-\alpha y^{2}+\beta y} dy=\left(\dfrac{\pi}{\alpha}\right)^{\frac{1}{2}}\exp\left(\dfrac{\beta^{2}}{4\alpha}\right).
\end{align*}
Moreover, we have the integral representation of the real Hermite polynomials,
\begin{align*}
H_{n}(x)=\dfrac{(2I)^{n}}{\sqrt{\pi}}\int_{-\infty}^{+\infty} e^{-(y+Ix)^{2}}y^{n}dy
\end{align*}
for every  $I\in \Sq $.
\end{lemma}

\begin{proof}[Proof of Theorem \ref{8}]
By the definition of $V$, we can write 
\begin{align*}
\mathcal{V}_{I}(h_{m},h_{n})(x+Iy)&=\dfrac{e^{-\frac{x^{2}+y^{2}}{4}}}{\sqrt{2}\pi}
\int_{\R}e^{-(t-I\frac{y}{2})^{2}}H_{m}\left(t+\frac{x}{2}\right) H_{n}\left(t-\frac{x}{2}\right)dt.
\end{align*}
By means of the generating function of the real Hermite polynomials $H_{m}$, we get
\begin{align*}
\sum_{m,n=0}^{+\infty} \frac{u^{m}}{m!}\frac{v^{m}}{n!}\mathcal{V}_{I}(h_{m},h_{n})(x+Iy)
&=\dfrac{e^{-\frac{x^{2}+y^{2}}{4}}}{\sqrt{2}\pi}e^{-u^{2}-v^{2}+(u-v)x}\int_{\R}
e^{-(t-I\frac{y}{2})^{2}} e^{2(u+v)t}dt\\
&=\dfrac{e^{-\frac{x^{2}+y^{2}}{4}}}{\sqrt{2}\pi}e^{-u^{2}-v^{2}+(u-v)x}
\sum_{k=0}^{+\infty}\dfrac{(2^k(u+v)^{k}}{k!}\int_{\R}t^{k}
e^{-(t-I\frac{y}{2})^{2}}dt.
\end{align*}
Next, Lemma \ref{1888} infers
\begin{align*}
\sum_{m,n=0}^{+\infty} \frac{u^{m}}{m!}\frac{v^{m}}{n!}\mathcal{V}_{I}(h_{m},h_{n})(x+Iy)&=
\dfrac{e^{-\frac{p^{2}+s^{2}}{4}}}{\sqrt{2}\pi}e^{-u^{2}-v^{2}+(u-v)x}
\sum_{k=0}^{+\infty}\dfrac{(u+v)^{k}}{I^{k}k!}H_{k}(\frac{y}{2})\\
&=\dfrac{e^{-\frac{x^{2}+y^{2}}{4}}}{\sqrt{2}\pi}
\exp\left(2uv+\sqrt{2}u\left(\frac{x-Iy}{\sqrt{2}}\right)
-\sqrt{2}v\left(\frac{x+Iy}{\sqrt{2}}\right)\right).
\end{align*}
In the right hand-side of the last equality we recognize the generating function of the quaternionic Hermite polynomials with variable $\dfrac{x+Iy}{\sqrt{2}}.$
Then, by identifying the two power series we get the result.
\end{proof}

We conclude by the following result whose the proof is straightforward.

\begin{theorem} \label{thm:FourierWigner-SBT}
The quaternionic Bargmann transform $\mathcal{B}_m$ leads to the quaternionic Fourier-Wigner transform $\mathcal{V}_{I}(f,\varphi)$, where $\varphi$ is given by the function $\varphi(t)=e^{-\frac{t^{2}}{2}}H_{m}(-t)$, in fact we have:
\begin{align*}
\mathcal{V}_{I}(f,\varphi)(q)=\sqrt{\pi}m! 2^{\frac{m-1}2}e^{-\frac{|q|^{2}}{4}}[\mathcal{B}_mf]\left(\overline{\dfrac{p+Is}{\sqrt{2}}}\right).
\end{align*}
\end{theorem}

{\bf\it Acknowledgements.}
 The present investigation was completed during the second-named author's visit to departimenti di Mathematica of Politecnico di Milano May - June 2017.
He would like to express his gratitude to Professor I.M. Sabadini for hospitality and many interesting discussions.
The authors are supported, in part, by the Hassan II Academy of Sciences and Technology.
The research work of A.G. was partially supported by a grant from the Simons Foundation.

\end{document}